\documentclass{article}%
\usepackage{graphicx}
\usepackage{amsmath}
\usepackage{amsfonts}
\usepackage{amssymb}
\usepackage{setspace}
\usepackage[hmargin=3.5cm,vmargin=3.5cm]{geometry}%
\providecommand{\U}[1]{\protect\rule{.1in}{.1in}}
\providecommand{\U}[1]{\protect\rule{.1in}{.1in}}
\providecommand{\U}[1]{\protect\rule{.1in}{.1in}}
\providecommand{\U}[1]{\protect\rule{.1in}{.1in}}
\providecommand{\U}[1]{\protect\rule{.1in}{.1in}}
\providecommand{\U}[1]{\protect\rule{.1in}{.1in}}
\providecommand{\U}[1]{\protect\rule{.1in}{.1in}}
\providecommand{\U}[1]{\protect\rule{.1in}{.1in}}
\providecommand{\U}[1]{\protect\rule{.1in}{.1in}}
\providecommand{\U}[1]{\protect\rule{.1in}{.1in}}
\providecommand{\U}[1]{\protect\rule{.1in}{.1in}}
\providecommand{\U}[1]{\protect\rule{.1in}{.1in}}
\providecommand{\U}[1]{\protect\rule{.1in}{.1in}}
\providecommand{\U}[1]{\protect\rule{.1in}{.1in}}
\providecommand{\U}[1]{\protect\rule{.1in}{.1in}}
\providecommand{\U}[1]{\protect\rule{.1in}{.1in}}
\providecommand{\U}[1]{\protect\rule{.1in}{.1in}}
\providecommand{\U}[1]{\protect\rule{.1in}{.1in}}
\providecommand{\U}[1]{\protect\rule{.1in}{.1in}}
\providecommand{\U}[1]{\protect\rule{.1in}{.1in}}
\providecommand{\U}[1]{\protect\rule{.1in}{.1in}}
\providecommand{\U}[1]{\protect\rule{.1in}{.1in}}
\providecommand{\U}[1]{\protect\rule{.1in}{.1in}}
\providecommand{\U}[1]{\protect\rule{.1in}{.1in}}
\providecommand{\U}[1]{\protect\rule{.1in}{.1in}}
\providecommand{\U}[1]{\protect\rule{.1in}{.1in}}
\providecommand{\U}[1]{\protect\rule{.1in}{.1in}}
\providecommand{\U}[1]{\protect\rule{.1in}{.1in}}
\providecommand{\U}[1]{\protect\rule{.1in}{.1in}}
\providecommand{\U}[1]{\protect\rule{.1in}{.1in}}
\providecommand{\U}[1]{\protect\rule{.1in}{.1in}}
\providecommand{\U}[1]{\protect\rule{.1in}{.1in}}
\providecommand{\U}[1]{\protect\rule{.1in}{.1in}}
\providecommand{\U}[1]{\protect\rule{.1in}{.1in}}
\providecommand{\U}[1]{\protect\rule{.1in}{.1in}}
\providecommand{\U}[1]{\protect\rule{.1in}{.1in}}
\providecommand{\U}[1]{\protect\rule{.1in}{.1in}}
\providecommand{\U}[1]{\protect\rule{.1in}{.1in}}
\providecommand{\U}[1]{\protect\rule{.1in}{.1in}}
\providecommand{\U}[1]{\protect\rule{.1in}{.1in}}

\newtheorem{theorem}{Theorem}
{}

\newtheorem{corollary}{Corollary}

\newtheorem{lemma}{Lemma}
{}

\newtheorem{proposition}{Proposition}

\newenvironment{proof}[1][Proof]{\textbf{#1.} }{\ \rule{0.5em}{0.5em}}

\begin{document}

\title{Spectral Problems of a Class of Non-self-adjoint One-dimensional Schrodinger
Operators }
\author{O. A. Veliev\\{\small Depart. of Math., Dogus University, Ac\i badem, Kadik\"{o}y, \ }\\{\small Istanbul, Turkey.}\ {\small e-mail: oveliev@dogus.edu.tr}}
\date{}
\maketitle

\begin{abstract}
In this paper we investigate the one-dimensional Schrodinger operator $L(q)$
with complex-valued periodic potential \ $q$ when $q\in L_{1}[0,1]$ \ and
$q_{n}=0$ for $n=0,-1,-2,...$, where $q_{n}$ are the Fourier coefficients of
$q$ with respect to the system $\{e^{i2\pi nx}\}.$ We prove that the Bloch
eigenvalues are $(2\pi n+t)^{2}$ for $n\in\mathbb{Z}$, $t\in\mathbb{C}$ and
find explicit formulas for \ the Bloch functions. Then we consider the inverse
problem for this operator.

Key Words: Hill operator, Spectrum, Inverse problems.

AMS Mathematics Subject Classification: 34L05, 34L20.

\end{abstract}

\section{Introduction and Preliminary Facts}

Let $L(q)$ be the operator generated in $L_{2}(-\infty,\infty)$ by the expression%

\begin{equation}
-y^{^{\prime\prime}}(x)+q(x)y(x)
\end{equation}
with a complex-valued periodic potential $q.$ In 1980, Gasymov [4] proved the
following remarkable results for the operator $L(q)$ with the potential $q$ of
the form \
\begin{equation}
q(x)=%
%TCIMACRO{\tsum \limits_{n=1}^{\infty}}%
%BeginExpansion
{\textstyle\sum\limits_{n=1}^{\infty}}
%EndExpansion
q_{n}e^{inx},
\end{equation}
where
\[%
%TCIMACRO{\tsum \limits_{n}}%
%BeginExpansion
{\textstyle\sum\limits_{n}}
%EndExpansion
\mid q_{n}\mid<\infty.
\]

\textbf{Result 1:} The spectrum $S(L(q))$ of the operator $L(q)$ is purely
continuous and
\begin{equation}
S(L(q))=[0,\infty).
\end{equation}
There may be second order spectral singularity on the continuous spectrum
which must coincide with numbers of the form $(\frac{n}{2})^{2}.$

\textbf{Result 2:} The equation%

\begin{equation}
-y^{^{\prime\prime}}(x)+q(x)y(x)=\mu^{2}y(x)
\end{equation}
has a solution of the form
\begin{equation}
f(x,\mu)=e^{i\mu x}(1+%
%TCIMACRO{\tsum \limits_{n=1}^{\infty}}%
%BeginExpansion
{\textstyle\sum\limits_{n=1}^{\infty}}
%EndExpansion
\frac{1}{n+2\mu}%
%TCIMACRO{\tsum \limits_{\alpha=n}^{\infty}}%
%BeginExpansion
{\textstyle\sum\limits_{\alpha=n}^{\infty}}
%EndExpansion
v_{n,\alpha}e^{i\alpha x}),
\end{equation}
where the following series converge
\[%
%TCIMACRO{\tsum \limits_{n=1}^{\infty}}%
%BeginExpansion
{\textstyle\sum\limits_{n=1}^{\infty}}
%EndExpansion
\frac{1}{n}%
%TCIMACRO{\tsum \limits_{\alpha=n+1}^{\infty}}%
%BeginExpansion
{\textstyle\sum\limits_{\alpha=n+1}^{\infty}}
%EndExpansion
\alpha(\alpha-n)\mid v_{n,\alpha}\mid,\text{ }%
%TCIMACRO{\tsum \limits_{n=1}^{\infty}}%
%BeginExpansion
{\textstyle\sum\limits_{n=1}^{\infty}}
%EndExpansion
n\mid v_{n,\alpha}\mid.
\]

\textbf{Result 3:} By the Floquet solutions (5) a spectral expansion was constructed.

\textbf{Result 4:} It was shown that the Wronskian of the Floquet solutions%
\[
f_{n}(x)=:\lim_{\mu\rightarrow\frac{n}{2}}(n-2\mu)f(x,-\mu)
\]
and $f(x,\frac{n}{2})$ is equal to zero and therefore they are linearly
dependent:
\begin{equation}
f_{n}(x)=s_{n}f(x,\frac{n}{2}).
\end{equation}
It was proved that from the generalized norming numbers $\{s_{n}\}$ one can
effectively reconstruct $\{q_{n}\}.$ That is, the inverse spectral problem was considered.

Guillemin and Uribe [6] investigated the boundary value problem (bvp)
generated on $[0,2\pi]$ by (1) and the periodic boundary conditions when $q\in
Q_{2}^{+},$ that is, $q\in L_{2}[0,2\pi]$ and has the form (2). It was proved
that the eigenvalues of this bvp are $n^{2}$ for $n\in\mathbb{Z}$ and the
corresponding root functions were studied. For the operator $L(q)$ with the
potential $q\in Q_{2}^{+}$ the inverse spectral problem was investigated in
detail by Pastur and Tkachenko [9] and the alternative proofs of (3) were
provided by Shin [10], Carlson [1] and Christiansen [2].

In this paper, we first prove that if $q\in L_{1}[0,1],$ $q(x+1)=q(x)$ and
\begin{equation}
\text{ }q_{n}=0,\text{ }\forall n=0,-1,-2,...,
\end{equation}
where $q_{n}=(q(x),e^{i2\pi nx})$ and $(.,.)$ is the inner product in
$L_{2}[0,1],$ then%
\begin{equation}
S(L(q))=[0,\infty),\text{ }S(L_{t}(q))=\{(2\pi n+t)^{2}:n\in\mathbb{Z}\}
\end{equation}
for all $t\in\mathbb{C},$ where $L_{t}(q)$ is the operator generated in
$L_{2}[0,1]$ by (1) and the conditions
\begin{equation}
y(1)=e^{it}y(0),\text{ }y^{^{\prime}}(1)=e^{it}y^{^{\prime}}(0).
\end{equation}
It is well-known that (see [3, 8]) the spectrum $S(L(q))$ of the operator
$L(q)$ is the union of the spectra $S(L_{t}(q))$ of the operators $L_{t}(q)$
for $t\in(-\pi,\pi].$ Thus we prove (8) for more general case and as one can
see from Theorem 1 that in a simple way. Moreover, we find explicit formulas
for the\ Bloch functions and consider the inverse problem for this general
case. The method of this paper is based on the following statements of my
paper [12]:

\textbf{ }\textit{The large eigenvalues }$\lambda_{n}(t)$\textit{ \ and the
corresponding normalized eigenfunctions }$\Psi_{n,t}(x)$\textit{ of the
operator }$L_{t}(q)$\textit{ for }$q\in L_{1}[0,1]$\textit{ and }$t\neq0,\pi
,$\textit{ satisfy the following asymptotic formulas }%
\begin{equation}
\lambda_{n}(t)=(2\pi n+t)^{2}+O(\frac{\ln\left\vert n\right\vert }{n}),\text{
}\Psi_{n,t}(x)=e^{i(2\pi n+t)x}+O(\frac{1}{n}).
\end{equation}
\textit{These asymptotic formulas are uniform with respect to }$t$\textit{ in
}$[\rho,\pi-\rho],$ \textit{where }$\rho\in(0,\frac{\pi}{2})$\textit{ (see
Theorem 2 of [9]). Furthermore, the following formulas hold (see (22) and (28)
in [12]): }%
\begin{equation}
(\lambda_{n}(t)-(2\pi n+t)^{2})(\Psi_{n,t},e^{i(2\pi n+t)x})=A_{m}(\lambda
_{n}(t))(\Psi_{n,t},e^{i(2\pi n+t)x})+R_{m+1}(\lambda_{n}(t)),
\end{equation}
\textit{where }%
\begin{equation}
A_{m}(\lambda)=\sum_{k=1,2,...,m}a_{k}(\lambda),
\end{equation}%
\begin{equation}
a_{k}(\lambda)=\sum_{n_{1},n_{2},...,n_{k}}\frac{q_{n_{1}}q_{n_{2}}%
...q_{n_{k}}q_{-n(k)}}{%
%TCIMACRO{\tprod \limits_{s=1,2,...,k}}%
%BeginExpansion
{\textstyle\prod\limits_{s=1,2,...,k}}
%EndExpansion
[\lambda-(2\pi(n-n(s))+t)^{2}]},
\end{equation}%
\begin{equation}
R_{m}(\lambda)=\sum_{n_{1},n_{2},...,n_{m}}\frac{q_{n_{1}}q_{n_{2}}%
...q_{n_{m}}(q\Psi_{n,t},e^{i(2\pi(n-n(m))+t)x})}{%
%TCIMACRO{\tprod \limits_{s=1,2,...,m}}%
%BeginExpansion
{\textstyle\prod\limits_{s=1,2,...,m}}
%EndExpansion
[\lambda-(2\pi(n-n(s))+t)^{2}]}=O(\frac{\ln\left\vert n\right\vert }{n})^{m},
\end{equation}%
\begin{equation}
n(s)=:n_{1}+n_{2}+\cdots+n_{s}%
\end{equation}
\textit{and the summations in (13) and (14) are taken with the conditions
}$n(s)\neq0$\textit{ for }$s=1,2,...$\textit{.}

\section{ On the Bloch Eigenvalues and Bloch Functions}

Denote by $L_{1}^{+}[0,1]$ and $L_{1}^{-}[0,1]$ the set of all $q\in
L_{1}[0,1]$ satisfying (7) and $q_{n}=0$ \ for $n=0,1,2,...$respectively. The
formula (11) immediately give us the following

\begin{theorem}
If $q\in L_{1}^{+}[0,1]$ then the eigenvalues of $L_{t}(q)$ for $t\in
\mathbb{C}$ are $(2\pi n+t)^{2}$, where $n\in\mathbb{Z}$. These eigenvalues
for $t\neq\pi k,$ where $k\in\mathbb{Z}$, are simple. The eigenvalues $(2\pi
n)^{2}$ for $n\in\mathbb{Z}\backslash\{0\}$ and $(2\pi n+\pi)^{2}$ for
$n\in\mathbb{Z}$ are double eigenvalues of $L_{0}(q)$ and $L_{\pi}(q)$
respectively. The theorem continues to hold if $L_{1}^{+}[0,1]$ is replaced by
$L_{1}^{-}[0,1].$
\end{theorem}

\begin{proof}
Since at least one of the indices $n_{1},n_{2},...,n_{k},$ $-n(k)$ (see (15))
is not positive number, by (7), (13) and (12) $a_{k}(\lambda_{n}(t))=0,$
$A_{m}(\lambda_{n}(t))=0$ $\ $for all $k,m.$Therefore in (11) letting $m$ tend
to infinity and then using (14) and (10) we obtain $\lambda_{n}(t)=(2\pi
n+t)^{2}$ for $t\in\lbrack\rho,\pi-\rho]$ and $n>N(\rho)\gg1.$ \ On the other
hand $\lambda_{n}(t)$ are the squares of the roots of
\[
F(\mu)=2\cos t,\text{ }\forall t\in\lbrack\rho,\pi-\rho],
\]
where $F(\mu)=\varphi^{^{\prime}}(1,\mu)+\theta(1,\mu)$ and $\varphi(x,\mu)$,
$\theta(x,\mu)$ are the solutions of the equation (4) satisfying the initial
conditions $\theta(0,\mu)=\varphi^{^{\prime}}(0,\mu)=1,\quad\theta^{^{\prime}%
}(0,\mu)=\varphi(0,\mu)=0$ (see [3]). Thus the entire functions $F(\mu)$ and
$2\cos\mu$ coincide on $\{(2\pi n+t):t\in\lbrack\rho,\pi-\rho]\}$. Therefore
these functions are identically equal in the complex plane and hence the
eigenvalues of $L_{t}(q)$ are the squares of the roots of the equation
$\cos\mu=\cos t$ for all $t\in\mathbb{C}.$ That is, in the case $q\in
L_{1}^{+}[0,1]$ the theorem is proved. The case $q\in L_{1}^{-}[0,1]$ can be
proved in the same way
\end{proof}

\begin{corollary}
Suppose $q\in L_{1}^{+}[0,1]\cup L_{1}^{-}[0,1]$. Let $E_{2n}$ and $E_{2n+1}$
be the subspaces corresponding to the eigenvalues $(2\pi n)^{2}$ and $(2\pi
n+\pi)^{2}$ of the operators $L_{0}(q)$ and $L_{\pi}(q)$ respectively. Then
any sequence obtained as the union of orthonolmal bases of all the subspaces
$E_{2n}$ ($E_{2n+1}$) for $n=0,1,...$ is a Reisz basis of $L_{2}[0,1].$
\end{corollary}

\begin{proof}
Since the periodic (antiperiodic) boundary conditions are regular by [11] the
sequence of subspaces $\left\{  E_{2n}\right\}  _{0}^{\infty}$ ($\left\{
E_{2n+1}\right\}  _{0}^{\infty}$) is a Riesz basis of the space $L_{2}[0,1]$.
Therefore the proof follows from the following well-known statement: \ If the
sequence of subspaces $\left\{  H_{n}\right\}  _{0}^{\infty}$ is a Riesz basis
of the Hilbert spaces $H,$ $\ $then any sequence obtained as the union of
orthonolmal bases of all the subspaces $H_{n}$ for $n=0,1,...$ form a Reisz
basis in $H$ ( see [5] p. 344).
\end{proof}

Note that the operators $L_{0}(q)$ and $L_{\pi}(q)$ in the case $q(x)=Ae^{2\pi
irx},$ where $A\in\mathbb{C}$ and $r\in\mathbb{Z}$, was investigated in detail
by N. B. Kerimov [7]. He found a necessary and sufficient condition for a
system of root functions of these operators to be a basis in $L_{p}[0,1]$ for
arbitrary $p\in(1,\infty).$ Moreover he determined whether the eigenvalue
$(\pi n)^{2}$ corresponds to the 2 linearly independent eigenfunctions or
eigenfunction and associated function and wrote explicit formulas for all root functions.

Now to consider the Bloch functions $\Psi_{n,t}(x)$ corresponding to the
eigenvalue $(2\pi n+t)^{2}$ we use the equality
\begin{equation}
((2\pi n+t)^{2}-(2\pi(n+p)+t)^{2})(\Psi_{n,t},e^{i(2\pi(n+p)+\overline{t}%
)x})=(q\Psi_{n,t},e^{i(2\pi(n+p)+\overline{t})x})
\end{equation}
obtained from $-\Psi_{n,t}^{^{\prime\prime}}+q\Psi_{n,t}=(2\pi n+t)^{2}%
\Psi_{n,t}$ by multiplying $e^{i(2\pi(n+p)+\overline{t})x}$.

\begin{theorem}
Suppose $q\in L_{1}^{+}[0,1]$. Let $\Psi_{n,t}(x)$ be the eigenfunction of the
operator $L_{t}(q)$ corresponding to the eigenvalue $(2\pi n+t)^{2}$ and
normalized as
\begin{equation}
(\Psi_{n,t},e^{i(2\pi n+\overline{t})x})=1,
\end{equation}
where $t\neq\pi k$ \ for $k\in\mathbb{Z}$. Then\textit{ }%
\begin{equation}
\text{ }\Psi_{n,t}(x)=e^{i(2\pi n+t)x}+\sum_{p\in\mathbb{N}}c_{p,n}%
(t)e^{i(2\pi(n+p)+t)x},
\end{equation}
where
\begin{equation}
c_{p,n}(t)=d_{p,n}(t)(q_{p}+\sum_{k=1}^{p-1}\sum_{n_{1},n_{2},...,n_{k}}%
\frac{q_{n_{1}}q_{n_{2}}...q_{n_{k}}q_{p-n(k)}}{%
%TCIMACRO{\tprod \limits_{s=1}^{k}}%
%BeginExpansion
{\textstyle\prod\limits_{s=1}^{k}}
%EndExpansion
(2\pi(2n+p-n(s))+2t)2\pi(n(s)-p)}),
\end{equation}
$d_{p,n}(t)=-(2\pi p(2\pi(2n+p)+2t))^{-1}$ for $p=1,2,...$ and
\begin{equation}
\{n_{1},n_{2},...,n_{s},p-n_{1}-n_{2}-...-n_{s}\}\subset\mathbb{N}%
=:\{1,2,...,\}
\end{equation}
for $s=1,2,...,p-1.$ The theorem continues to hold if $L_{1}^{+}[0,1]$ is
replaced by $L_{1}^{-}[0,1]$ and $\mathbb{N}$ in (18) and (20) is replaced by
$-\mathbb{N}$\textit{.}
\end{theorem}

\begin{proof}
Let $\Psi_{n,t}(x)$ be the normalized eigenfunction of the operator $L_{t}(q)$
corresponding to the eigenvalue $(2\pi n+t)^{2}$ and $t\neq\pi k$ \ for
$k\in\mathbb{Z}.$ (In the end we prove that there exists an eigenfunction of
the operator $L_{t}(q)$ satisfying (17). For simplicity of notation we denote
it also by $\Psi_{n,t}$). Since the systems $\{e^{i(2\pi n+t)x}:n\in
\mathbb{Z}\}$ and $\{e^{i(2\pi n+\overline{t})x}:n\in\mathbb{Z}\}$ are
biorthogonal in $L_{2}[0,1]$ we have
\begin{equation}
\Psi_{n,t}(x)-(\Psi_{n,t},e^{i(2\pi n+\overline{t})x})e^{i(2\pi n+t)x}%
=\sum_{p\in\mathbb{Z}\backslash\{0\}}(\Psi_{n,t},e^{i(2\pi(n+p)+\overline
{t})x})e^{i(2\pi(n+p)+t)x}.
\end{equation}
To find $(\Psi_{n,t},e^{i(2\pi(n+p)+\overline{t})x})$ we iterate (16) by
using
\begin{equation}
(q\Psi_{n,t},e^{i(2\pi(n+p)+\overline{t})x})=\sum_{n_{1}}q_{n_{1}}(\Psi
_{n,t}(x),e^{i(2\pi(n+p-n_{1})+\overline{t})x})
\end{equation}
(see (14) of [12]). Namely, using (22) and (16) we get
\begin{equation}
(\Psi_{n,t},e^{i(2\pi(n+p)+\overline{t})x})=d_{p,n}(t)\sum_{n_{1}}q_{n_{1}%
}(\Psi_{n,t},e^{i(2\pi(n+p-n_{1})+\overline{t})x}).
\end{equation}
Now isolate the terms in the right-hand side of (23) containing the multiplicand

$(\Psi_{n,t},e^{i(2\pi n+\overline{t})x})$ which occurs in the case $n_{1}=p$
and use (23) for the other terms to get
\[
(\Psi_{n,t},e^{i(2\pi(n+p)+\overline{t})x})=d_{p,n}(t)(q_{p}(\Psi
_{n,t},e^{i(2\pi n+\overline{t})x})+\sum_{n_{1},n_{2}}\frac{q_{n_{1}}q_{n_{2}%
}(\Psi_{n,t},e^{i(2\pi(n+p-n_{1}-n_{2})+\overline{t})x})}{(2\pi(2n+p-n_{1}%
)+2t)2\pi(n_{1}-p)}).
\]
Repeating this process $m$ times, that is, isolating again the terms
containing the multiplicand $(\Psi_{n,t},e^{i(2\pi n+\overline{t})x})$ which
occurs in the case $n_{1}+n_{2}=p$ and using again (23) for the other terms
and doing this iteration $m$ times, we obtain%
\begin{equation}
(\Psi_{n,t},e^{i(2\pi(n+p)+\overline{t})x})=c_{p,n}(t)(\Psi_{n,t},e^{i(2\pi
n+\overline{t})x})+r_{m},
\end{equation}%
\begin{equation}
r_{m}=d_{p,n}(t)\sum_{n_{1},n_{2},...,n_{m}}\left(  \frac{q_{n_{1}}q_{n_{2}%
}...q_{n_{m}}(q\Psi_{n,t},e^{i(2\pi(n+p-n(m))+\overline{t})x})}{%
%TCIMACRO{\tprod \limits_{s=1}^{m}}%
%BeginExpansion
{\textstyle\prod\limits_{s=1}^{m}}
%EndExpansion
[(2\pi(2n+p-n(s))+2t)2\pi(n(s)-p)]}\right)  ,
\end{equation}
where $m>p,$ $p-n(s)\neq0$ for $s=1,2,...,m.$

Now we prove that $r_{m}\rightarrow0$ as $m\rightarrow\infty.$ By (20),
$n_{k}\geq1$ for $k=1,2,...,m$ and hence $n(s)\geq s.$ Using this and taking
into account that \ $(q\Psi_{n,t},e^{i(2\pi(n+p-n(m))+\overline{t}%
)x})\rightarrow0$ as $m\rightarrow\infty$ from (25) we obtain
\begin{equation}
\mid r_{m}\mid\leq\mid d_{p,n}(t)\mid%
%TCIMACRO{\tprod \limits_{s=1,2,...,m}}%
%BeginExpansion
{\textstyle\prod\limits_{s=1,2,...,m}}
%EndExpansion
\left(  \sum_{j\geq s,\text{ }j\neq p}\frac{M}{\mid(2\pi(2n+p-j)+2t)2\pi
(j-p)\mid}\right)
\end{equation}
for $m\gg1$, where $M=\sup_{n}\mid q_{n}\mid.$ Clearly, there exists $K(t)$
such that
\begin{equation}
\sum_{j\geq s,\text{ }j\neq p}\mid\frac{M}{(2\pi(2n+p-j)+2t)2\pi(j-p)}\mid\leq
K(t).
\end{equation}
for $s=1,2,...,m.$ Moreover, if $s>4(\mid n\mid+\mid p\mid)$ then%
\begin{equation}
\sum_{j\geq s}\mid\frac{M}{(2\pi(2n+p-j)+2t)2\pi(j-p)}\mid<\sum_{j\geq s}%
\frac{M}{j^{2}}<\frac{M}{s-1}.
\end{equation}
Now using (26)-(28) we obtain
\[
\mid r_{m}\mid\leq\frac{\mid d_{p,n}(t)\mid M^{m-4(\mid n\mid+\mid p\mid
)}(K(t))^{4(\mid n\mid+\mid p\mid)}}{4(\mid n\mid+\mid p\mid)(4(\mid
n\mid+\mid p\mid)+1)...(m-1)}%
\]
which implies that $r_{m}\rightarrow0$ as $m\rightarrow\infty.$ Therefore in
(24) letting $m$ tend to infinity we get
\begin{equation}
(\Psi_{n,t},e^{i(2\pi(n+p)+\overline{t})x})=c_{p,n}(t)(\Psi_{n,t},e^{i(2\pi
n+\overline{t})x}).
\end{equation}
This with (21) shows that $(\Psi_{n,t},e^{i(2\pi n+\overline{t})x})\neq0.$
Therefore, there exists eigenfunction, denoted again by $\Psi_{n,t},$
satisfying (17) and for this eigenfunction, by (29), we have
\begin{equation}
(\Psi_{n,t},e^{i(2\pi(n+p)+\overline{t})x})=c_{p,n}(t).
\end{equation}
The indices $n_{1},n_{2},...,n_{k},p-n(k)$ taking part in the expression of
$c_{p,n}(t)$ (see (19)) satisfy (20). Therefore if $p<0,$ then the set of
these indices is empty, that is, the first term on the right-hand side of (24)
does not appear at all. Hence, from (24) using the relation $r_{m}%
\rightarrow0$ as $m\rightarrow\infty$ we obtain
\begin{equation}
(\Psi_{n,t},e^{i(2\pi(n+p)+\overline{t})x})=0,\text{ }\forall p<0.
\end{equation}
Thus (18) for $q\in L_{1}^{+}[0,1]$ follows from (21), (17), (30) and (31).
The case $q\in L_{1}^{-}[0,1]$ can be considered in the same way.
\end{proof}

\section{On the Inverse Problem}

First consider the Floquet solutions of the equation (4) defined by%

\begin{equation}
\Psi(x,\mu)=\Psi_{n,t}(x)
\end{equation}
for $\mu=(2\pi n+t),$ where $n\in\mathbb{Z}$, $\operatorname{Re}t\in(-\pi
,\pi]$ and $\Psi_{n,t}(x)$ is studied in Theorem 2. Since $\Psi_{n,t}(x)$
satisfies (9), we have
\[
\Psi_{n,t}(x+m)=e^{itm}\Psi_{n,t}(x)
\]
for all $x\in(-\infty,\infty)$ and $m\in\mathbb{Z}$. This with the equality
$\operatorname{Im}\mu=\operatorname{Im}t$ implies that

$\Psi(x,\mu)\in L_{2}(a,\infty),$ $\Psi(x,-\mu)\in L_{2}(-\infty,a)$ for
$\operatorname{Im}\mu>0$ and $a\in(-\infty,\infty)$. Therefore repeating the
arguments of [4] one can obtain the spectral expansion. Note that we construct
the Floquet solution for more general case and by the other method (see Result
2 in introduction).

Now we consider the inverse problem as follows. We write the Fourier
decomposition of $\Psi_{n,t}(x)$ and $\Psi_{-n,-t}(x)$ in the form \ \textit{
}%
\begin{equation}
\text{ }\Psi_{n,t}(x)=\sum_{p\in\mathbb{Z}}c_{p,n}(t)e^{i(2\pi(n+p)+t)x}%
,\text{ }\Psi_{-n,-t}(x)=\sum_{p\in\mathbb{Z}}c_{p,-n}(-t)e^{i(2\pi
(-n+p)-t)x},
\end{equation}
\ \textit{ }where, by Theorem 2, $c_{p,n}(t)$ for $p>0$ is defined by (19)
and
\begin{equation}
c_{0,n}(t)=1,\text{ }c_{p,n}(t)=0,\text{ }\forall p<0.
\end{equation}
First we show that
\begin{equation}
\lim_{t\rightarrow0}8nt\pi c_{2n+p,-n}(-t)=c_{p,n}(0)s_{2n},\text{ }\forall
p\geq-2n,
\end{equation}%
\begin{equation}
\text{ }\lim_{t\rightarrow\pi}4(2n+1)(t-\pi)\pi c_{2n+p+1,-n}(-t)=c_{p,n}%
(\pi)s_{2n+1},\text{ }\forall p\geq-2n-1,
\end{equation}
\textit{ }where
\begin{equation}
s_{n}=q_{n}+\sum_{k=1}^{n-1}S_{k}(n),\text{ }S_{k}(n)=\sum_{n_{1}%
,n_{2},...,n_{k}}\frac{q_{n_{1}}q_{n_{2}}...q_{n_{k}}q_{n-n(k)}}{%
%TCIMACRO{\tprod \limits_{s=1}^{k}}%
%BeginExpansion
{\textstyle\prod\limits_{s=1}^{k}}
%EndExpansion
(2\pi n(s))2\pi(n-n(s))},\text{ }n=1,2,...
\end{equation}
(see Lemma 1). Then using these equalities we prove that
\begin{equation}
\lim_{t\rightarrow0}8n\pi t\Psi_{-n,-t}(x)=s_{2n}\Psi_{n,0}(x),\forall n\geq1,
\end{equation}%
\begin{equation}
\lim_{t\rightarrow\pi}4(2n+1)(t-\pi)\pi\Psi_{-n,-t}(x)=s_{2n+1}\Psi_{n,\pi
}(x),\forall n\geq0
\end{equation}
(see Lemma 2 and Theorem 3). Due to (32) and (6) $\{s_{n}\}$ is the sequence
of the norming numbers. Finally, we investigate the property of the norming
numbers and consider the question when $\{s_{n}\}$ may be a sequence of the
norming number for the operator $L(q)$ with potential $q\in L_{1}^{+}[0,1]$
(see Theorem 4, Proposition 1 and Corollary 2).

\begin{lemma}
The equalities (35) and (36) hold for all $n\geq1$ and $n\geq0$ respectively.
\end{lemma}

\begin{proof}
The proof of (35) for $p=-2n$ follows from (34). From the definition of
$d_{p,n}(t)$ (see Theorem 2) we see that
\begin{equation}
d_{2n+p,-n}(0)=d_{p,n}(0),\text{ }\lim_{t\rightarrow0}8n\pi td_{2n,-n}(-t)=1,
\end{equation}%
\begin{equation}
\lim_{t\rightarrow0}8n\pi td_{2n+p,-n}(-t)=0,\text{ }\forall p\neq0,-2n.
\end{equation}
Therefore by (19) and (37) we have
\[
\lim_{t\rightarrow0}8n\pi tc_{2n,-n}(-t)=s_{2n}.
\]
Thus the proof of (35) for $p=0$ also follows from (34).

To prove (35) in the more complicated cases $p\neq0,-2n$ we rewrite
$c_{2n+p,-n}(-t),$ $c_{p,n}(0)$ and $s_{2n}$ in the following form
\begin{equation}
c_{2n+p,-n}(-t)=\sum_{k=0}^{2n+p-1}D_{k}(-t),\text{ }D_{0}(-t)=d_{2n+p,-n}%
(-t)q_{2n+p},
\end{equation}%
\begin{equation}
D_{k}(-t)=\sum_{n_{1},n_{2},...,n_{k}}D_{k}(-t,n_{1},n_{2},...,n_{k}),\text{
}\forall k>0,
\end{equation}%
\begin{equation}
D_{k}(-t,n_{1},n_{2},...,n_{k})=\frac{d_{2n+p,-n}(-t)q_{n_{1}}q_{n_{2}%
}...q_{n_{k}}q_{2n+p-n(k)}}{%
%TCIMACRO{\tprod \limits_{s=1}^{k}}%
%BeginExpansion
{\textstyle\prod\limits_{s=1}^{k}}
%EndExpansion
(2\pi(p-n(s))-2t)2\pi(n(s)-2n-p)},
\end{equation}%
\begin{equation}
c_{p,n}(0)=\sum_{j=0}^{p-1}F_{j},\text{ }F_{0}=d_{p,n}(0)q_{p},
\end{equation}%
\begin{equation}
F_{j}=d_{p,n}(0)\sum_{n_{1},n_{2},...,n_{j}}\frac{q_{n_{1}}q_{n_{2}%
}...q_{n_{j}}q_{p-n(j)}}{%
%TCIMACRO{\tprod \limits_{s=1}^{j}}%
%BeginExpansion
{\textstyle\prod\limits_{s=1}^{j}}
%EndExpansion
(2\pi(p-n(s)))2\pi(n(s)-2n-p)},
\end{equation}%
\begin{equation}
s_{2n}=\sum_{i=0}^{2n-1}E_{i},\text{ }E_{0}=q_{2n},\text{ }E_{i}=\sum
_{m_{1},m_{2},...,m_{i}}\text{ }E(m_{1},m_{2},...,m_{i}),\text{\ }%
\end{equation}%
\begin{equation}
E(m_{1},m_{2},...,m_{i})=\frac{q_{m_{1}}q_{m_{2}}...q_{m_{i}}q_{2n-m(i)}}{%
%TCIMACRO{\tprod \limits_{s=1}^{i}}%
%BeginExpansion
{\textstyle\prod\limits_{s=1}^{i}}
%EndExpansion
(2\pi m(s))2\pi(2n-m(s))}.
\end{equation}
By (41)
\begin{equation}
\lim_{t\rightarrow0}8n\pi tD_{0}(-t)=0.
\end{equation}

Now let us investigate $D_{k}(-t)$ for $1\leq k\leq2n+p-1$ and $p\neq0,-2n.$
By (44), (7) and by (15) and (20) we have $2n+p-n(k)>0$ and $n(s)<n(k)$ for
$s<k$ . Therefore the multiplicand $n(s)-2n-p$ of the denominator of the
fraction in (44) is a negative integer:
\begin{equation}
\text{ }n(s)-2n-p<0
\end{equation}
for $s=1,2,...,k.$ To investigate the other multiplicand $p-n(s)$ consider the cases:

Case 1: $-2n<p<0.$ Since $n(s)>0,$ we have $p-n(s)\neq0.$ This with (50)
gives
\[
\lim_{t\rightarrow0}8n\pi tD_{k}(-t)=0.
\]
Therefore (35) follows from (34), (42) and (49).

Case 2: $p>0.$ \ One can readily see that $D_{k}(-t)$ can be written in the
form%
\begin{equation}
D_{k}(-t)=\sum_{j=-1}^{p-1}D_{k,j}(-t),
\end{equation}
where $D_{k,-1}(-t)$ and $D_{k,j}(-t)$ for $j\geq0$ are the right-hand side of
(43) when the summation is taken under conditions
\begin{equation}
p-n(s)\neq0,\text{ }\forall s=1,2,...k
\end{equation}
and
\begin{equation}
n_{1}+n_{2}+...+n_{j+1}=p.
\end{equation}
respectively. By (52), (50) and (44) we have
\begin{equation}
\lim_{t\rightarrow0}8n\pi tD_{k,-1}(-t)=0.
\end{equation}

Now consider $D_{k,j}(-t)$ for $j\geq0,$ i.e., assume that (53) holds. The
indices $n_{1},n_{2},...,n_{j+1}$ satisfying (53) take part in $D_{k,j}(-t)$
if and only if $j+1\leq k\leq2n+j,$ since $n_{s}>0$ for all $s=1,2,...$ and
$2n+p-n(k)>0$ (see (44)). Therefore
\begin{equation}
D_{k,j}(-t)=0
\end{equation}
for $k\leq j$ and for $k>2n+j.$ Thus it remains to consider $D_{k,j}(-t)$ for
$j\geq0$ and

$j+1\leq k\leq2n+j.$ If (53) holds then $n_{j+1}=p-n(j)$ and by (44) the
expression $D_{k}(-t,n_{1},n_{2},...,n_{k})$ can be written as product of
\[
\frac{d_{2n+p,-n}(-t)q_{n_{1}}q_{n_{2}}...q_{n_{j}}q_{p-n(j)}}{\left(
%TCIMACRO{\tprod \limits_{s=1}^{j}}%
%BeginExpansion
{\textstyle\prod\limits_{s=1}^{j}}
%EndExpansion
(2\pi(p-n(s))-2t)2\pi(n(s)-2n-p)\right)  8n\pi t}%
\]
and
\[
\frac{q_{n_{j+2}}q_{n_{j+3}}...q_{n_{k}}q_{2n-n_{j+2}-n_{j+3}-...-n_{k}}}{%
%TCIMACRO{\tprod \limits_{s=j+2}^{k}}%
%BeginExpansion
{\textstyle\prod\limits_{s=j+2}^{k}}
%EndExpansion
(2\pi(n_{k+2}+n_{k+3}+...+n_{s}))2\pi(2n-n_{k+2}-n_{k+3}-...-n_{s})}.
\]
The last expression is $E(m_{1},m_{2},...,m_{k-j-1})$ for $m_{1}=n_{j+2},$
$m_{2}=n_{j+3},...,m_{k-j-1}=n_{k}$ (see (48)). Using this and (40) and taking
into account that $p-n(s)\neq0$ for all $s\neq j+1$ (see (53) and use the
inequality $n_{s}>0$ for all $s=1,2,...$ ) we obtain
\[
\lim_{t\rightarrow0}8n\pi tD_{k,j}(-t)=F_{j}\sum_{m_{1},m_{2},...,m_{k-j-1}%
}E(m_{1},m_{2},...,m_{k-j-1}).
\]
This with (55), (47) and (45) implies that
\[
\text{ }\lim_{t\rightarrow0}8n\pi t\sum_{k=0}^{2n+p-1}D_{k,j}(-t)=F_{j}%
\sum_{k=j+1}^{2n+j}E_{k-j-1}=F_{j}s_{2n},
\]%
\[
\text{ }\lim_{t\rightarrow0}8n\pi t\sum_{j=0}^{p-1}\sum_{k=0}^{2n+p-1}%
D_{k,j}=s_{2n}\sum_{j=0}^{p-1}F_{j}=s_{2n}c_{p,n}(0).
\]
Thus (35) follows from (42), (51) and (54). In the same way one can prove (36)
\end{proof}

\begin{lemma}
For any $n\geq1$ and $n\geq0$ there exist constants $K$ and $L$ such that the
inequalities
\begin{equation}
8n\pi t\mid(q\Psi_{-n,-t},e^{i(2\pi m-t)x})\mid<K,\text{ }\forall
m\in\mathbb{Z}%
\end{equation}
and
\begin{equation}
4(2n+1)\pi(\pi-t)\mid(q\Psi_{-n,-t},e^{i(2\pi m-t)x})\mid<L,\text{ }\forall
m\in\mathbb{Z}%
\end{equation}
hold for $t\in(0,\frac{\pi}{2})$ \ and for $t\in\lbrack\frac{\pi}{2},\pi)$ respectively.
\end{lemma}

\begin{proof}
Let us prove (56) for $t\in(0,\frac{\pi}{2}).$ Since $(q(x)\Psi_{-n,-t}%
(x),e^{i(2\pi m-t)x})$ tends to zero as $\left\vert m\right\vert
\rightarrow\infty,$ there exists a constant $C(t)$ and integer $k_{0}(t)$ such
that
\begin{equation}
\underset{m\in\mathbb{Z}}{max}\left\vert (q(x)\Psi_{-n,-t}(x),e^{i(2\pi
m-t)x})\right\vert =\left\vert (q(x)\Psi_{-n,-t}(x),e^{i(2\pi k_{0}%
-t)x})\right\vert =C(t).
\end{equation}
Let $l$ be an integer such that
\begin{equation}
\sum_{k\geq l}\frac{1}{k^{2}}<\frac{1}{2M},
\end{equation}
where $M$ is defined in (26). Using (58), (22) and then (30), (59) we obtain
\begin{equation}
C(t)=\left\vert (q\Psi_{-n,-t},e^{i(2\pi k_{0}-t)x})\right\vert \leq\left\vert
\sum_{m:\mid k_{0}-m\mid\leq\mid n\mid+l}q_{m}(\Psi_{-n,-t},e^{i(2\pi
(k_{0}-m)-t)x})\right\vert
\end{equation}%
\[
+\mid\sum_{m:\mid k_{0}-m\mid>\mid n\mid+l}\frac{q_{m}(q(x)\Psi_{N,t}%
(x),e^{i(2\pi(k_{0}-m)+t)x})}{(-2\pi n-t)^{2}-(2\pi(k_{0}-m)-t)^{2}}%
\mid<S+\frac{C(t)}{2},
\]
where
\begin{equation}
S=M\sum_{p:\mid p\mid\leq2\mid n\mid+l}\mid c_{-n,p}(-t)\mid.
\end{equation}
On the other hand, from (19) and (20) one can readily see that there exists a
constant $c$ such that $8n\pi t\left\vert c_{-n,p}(-t)\right\vert <c$ for all
$p$ with $\mid p\mid\leq2\mid n\mid+l.$ Moreover the number of the summands in
(61) is less that $2(2\mid n\mid+l+1).$ Therefore $8n\pi t\mid S\mid<2M(2\mid
n\mid+l+1)c.$ This inequality with (60) implies that
\[
8n\pi tC(t)<2M(2\mid n\mid+l+1)c+\frac{8n\pi tC(t)}{2},
\]
that is, (56) holds for $K=4M(2\mid n\mid+l+1)c$. In the same way we prove (57).
\end{proof}

\begin{theorem}
If $q\in L_{1}^{+}[0,1]$ then (38) and (39) hold for $n\geq1$ and $n\geq0$ respectively.
\end{theorem}

\begin{proof}
From (30), (16) and Lemma 2 it follows that
\begin{equation}
8n\pi t\left\vert c_{p,-n}(-t)\right\vert =\mid8n\pi td_{p,-n}(-t)(q\Psi
_{-n,-t},e^{i(2\pi(-n+p)-\overline{t})x})\mid<K\mid d_{p,-n}(t)\mid
\end{equation}
for $t\in(0,\frac{\pi}{2}).$ Therefore the series \textit{ }%
\begin{equation}
\text{ }8n\pi t\Psi_{-n,-t}(x)=8n\pi te^{i(-2\pi n-t)x}+\sum_{p\in\mathbb{N}%
}8n\pi tc_{p,-n}(-t)e^{i(2\pi(-n+p)-t)x})
\end{equation}
(see (18)) converges uniformly with respect to $x\in\lbrack0,1]$ and
$t\in(0,\frac{\pi}{2}].$ Thus, in (63) letting $t$ tend to zero and using
equality (35) we get the proof of (38). To prove (39) instead of (35) and (56)
we use (36) and (57) and repeat the proof of (38).
\end{proof}

From (38) and (39) we define the norming numbers $s_{n}$ for $n=1,2,...$By (37)%

\begin{equation}
s_{1}=q_{1},\text{ }s_{2}=q_{2}+\frac{q_{1}^{2}}{(2\pi)^{2}},\text{ }%
s_{3}=q_{3}+\frac{q_{1}q_{2}}{(2\pi)^{2}}+\frac{q_{1}^{3}}{4(2\pi)^{4}},....
\end{equation}
Thus if the norming numbers $s_{n}$ for $n=1,2,...$ are given then one can
define recurrently \ %

\begin{equation}
q_{1}=s_{1},\text{ }q_{2}=s_{2}-\frac{s_{1}^{2}}{(2\pi)^{2}},\text{ }%
q_{3}=s_{3}-\frac{s_{1}s_{2}}{(2\pi)^{2}}+\frac{3s_{1}^{3}}{4(2\pi)^{4}},....
\end{equation}

Now we are ready to prove the main result of this section. Let $S$ be the set
of all sequences $\{s_{n}\}$ for which there exists $s\in L_{1}[0,1]$ with
$(s(x),e^{i2\pi nx})=s_{n}$ for all $n=1,2,....$

\begin{theorem}
For every $q\in L_{1}^{+}[0,1]$ the sequence $\{s_{n}\}$ of norming numbers is
an element of $S.$ Conversely, for any sequence $\{s_{n}\}$ from $S$ there
exists unique $q\in L_{1}^{+}[0,1]$ such that the sequence of the norming
numbers of $L(q)$ coincides with $\{s_{n}\}$ if and only if the solutions
$q_{1},q_{2},...$ of (37) is a bounded sequence.
\end{theorem}

\begin{proof}
Since $\mid q_{n}\mid\leq M$ for all $n,$ where $M$ is defined in (26), from
(37) we obtain
\begin{equation}
\mid S_{1}(n)\mid\leq\frac{M^{2}}{(2\pi)^{2}}\sum_{k=1}^{n-1}\frac{1}{k(n-k)}.
\end{equation}
In the same way we get
\begin{equation}
\mid S_{k}(n)\mid\leq\frac{M^{k+1}}{(2\pi)^{2k}}\left(  \sum_{k=1}^{n-1}%
\frac{1}{k(n-k)}\right)  ^{k}.
\end{equation}
Now using the obvious inequality
\begin{equation}
\sum_{k=1}^{n-1}\frac{1}{k(n-k)}<\frac{2(1+\ln n)}{n},\text{ }\forall n\geq4
\end{equation}
from (37) we obtain that
\begin{equation}
s_{n}-q_{n}=O(n^{-1}\ln n).
\end{equation}
Therefore there exists $p\in L_{2}[0,1]$ such that $(p(x),e^{i2\pi nx}%
)=s_{n}-q_{n}$ for all $n=1,2,....$ Then the function $s(x)=p(x)+q(x)$ belongs
to $L_{1}[0,1]$ and $(s(x),e^{i2\pi nx})=s_{n}$ for all $n=1,2,....,$ that is,
$\{s_{n}\}\in S.$

Now suppose that $\{s_{n}\}\in S$ and the solutions $q_{1},q_{2},...$ of (37)
is a bounded sequence. Then there exists a constant $C$ such that $\mid
q_{n}\mid\leq C$ for $n=1,2,...$ Instead of $M$ using $C$ and repeating the
proof of (69) we see that $\{q_{n}\}\in S.$ Therefore there exists unique
$q\in L_{1}^{+}[0,1]$ such that the sequence of the norming numbers of $L(q)$
coincides with $\{s_{n}\}.$
\end{proof}

It remains to find the conditions on the sequence $\{s_{n}\}$ of norming
numbers such that the sequence $\{q_{n}\}$ defined from (37) is bounded. Below
we present an example by using the following obvious inequality
\begin{equation}
\sum_{k=1}^{n-1}\frac{1}{k(n-k)}\leq1,\text{ }\forall n>1.
\end{equation}
which follows from (68) for $n>5$ and can be verified by calculations for
$n\leq5.$

\begin{proposition}
If the sequence $\{s_{n}\}$ of norming numbers satisfies the inequality%
\begin{equation}
\mid s_{n}\mid\leq2\pi-\frac{2\pi}{2\pi-1},\text{ }\forall n=1,2,...
\end{equation}
then for the sequence $\{q_{n}\}$ defined from (37) the following estimations
hold%
\begin{equation}
\mid q_{m}\mid\leq2\pi,\text{ }\forall m=1,2,...
\end{equation}

\end{proposition}

\begin{proof}
Let us prove (72) \ by induction. It follows from (65) that (72) holds for
$m=1,2.$ Assume that (72) holds for $m<n,$ where $n>2.$ Then $\mid q_{m}%
\mid\leq2\pi$ for $m<n.$ The indices $n_{1},n_{2},...,n_{k},n-n_{k}$ taking
part in the expressions of $S_{k}(n)$ for $k=1,2,...,n-1$ (see (37)) less that
$n,$ since they are positive numbers and their total sum is $n.$ Therefore by
assumption of the induction the Fourier coefficients taking part in those
expressions satisfy (72). \ Thus in (67) instead of $M$ taking $2\pi$ and then
using (70) we get
\[
\sum_{k=1}^{n-1}\mid S_{k}(n)\mid\leq\sum_{k=1}^{n-1}\frac{(2\pi)^{k+1}}%
{(2\pi)^{2k}}<\frac{2\pi}{2\pi-1}.
\]
This with (37) and (71) implies that $\mid q_{n}\mid\leq\mid s_{n}\mid
+\frac{2\pi}{2\pi-1}\leq2\pi$
\end{proof}

Theorem 4 with Proposition 1 implies

\begin{corollary}
For any sequence $\{s_{n}\}\in S$ satisfying (71) there exists unique $q\in
L_{1}^{+}[0,1]$ such that the sequence of the norming numbers of $L(q)$
coincides with $\{s_{n}\}.$
\end{corollary}

\end{document}